\documentclass{amsart}

\usepackage{amsthm}
\usepackage[leqno]{amsmath}
\usepackage{latexsym,amsfonts,amssymb}
\usepackage[all]{xy} \SelectTips{eu}{} \SilentMatrices
\usepackage{hyperref}
\usepackage{verbatim}

\newcommand{\numberseries}{\mdseries}   %Fontseries used for numbering

\newlength{\thmtopspace}                %Space above theorem
\newlength{\thmbotspace}                %Space below theorem
\newlength{\thmheadspace}               %Space after theorem label
\newlength{\thmindent}                  %For indenting

\renewcommand{\subparagraph}{\vspace*{\thmbotspace}}

\newtheoremstyle{bfupright head,slanted body}
                {\thmtopspace}{\thmbotspace}
                {\slshape}{\thmindent}{\bfseries}{.}{\thmheadspace}
                {{\numberseries \thmnumber{(#2) }}\thmnote{#3}}

\newtheoremstyle{bfupright head,upright body}
                {\thmtopspace}{\thmbotspace}
                {\upshape}{\thmindent}{\bfseries}{.}{\thmheadspace}
                {{\numberseries \thmnumber{(#2) }}\thmnote{#3}}

\newtheoremstyle{bfit head,upright body}
                {\thmtopspace}{\thmbotspace}
                {\upshape}{\thmindent}{\upshape}{.}{\thmheadspace}
                {{\numberseries\thmnumber{(#2) }}
                {\bfseries\itshape\thmnote{\negthickspace#3}}}

\newtheoremstyle{it head,upright body}
                {\thmtopspace}{\thmbotspace}
                {\upshape}{\thmindent}{\upshape}{.}{\thmheadspace}
                {{\numberseries\thmnumber{(#2) }}
                {\itshape\thmnote{\negthickspace#3}}}

\newtheoremstyle{fixed bf head,slanted body}
                {\thmtopspace}{\thmbotspace}{\slshape}
                {\thmindent}{\bfseries}{.}{\thmheadspace}
                {{\numberseries \thmnumber{(#2) }}\thmname{#1}\thmnote{ (#3)}}

\newtheoremstyle{fixed bf head,upright body}
                {\thmtopspace}{\thmbotspace}{\upshape}
                {\thmindent}{\bfseries}{.}{\thmheadspace}
                {{\numberseries \thmnumber{(#2) }}\thmname{#1}\thmnote{ (#3)}}

\newtheoremstyle{fixed bfit head,upright body}
                {\thmtopspace}{\thmbotspace}{\upshape}
                {\thmindent}{\bfseries\itshape}{.}{\thmheadspace}
                {{\numberseries \thmnumber{(#2) }}\thmname{#1}\thmnote{ (#3)}}

\newtheoremstyle{sc head,small body}
                {\thmtopspace}{\thmbotspace}
                {\small\upshape}{\thmindent}{\scshape}{.}{\thmheadspace}
                {\thmname{#1}}

\newtheoremstyle{numbered paragraph}
                {\thmtopspace}{\thmbotspace}{\upshape}
                {\thmindent}{\upshape}{}{0pt}
                {{\numberseries \thmnumber{(#2) }}}

\newtheoremstyle{unnumbered paragraph}
                {\thmtopspace}{\thmbotspace}{\upshape}
                {\parindent}{\upshape}{}{0pt}

\theoremstyle{bfupright head,slanted body}
\newtheorem{res}{}[section]             \newtheorem*{res*}{}

\theoremstyle{bfit head,upright body}
                 \newtheorem*{com*}{}

\theoremstyle{bfupright head,upright body}
\newtheorem{bfhpg}[res]{}               \newtheorem*{bfhpg*}{}

\theoremstyle{it head,upright body}
               \newtheorem*{ithpg*}{}

\theoremstyle{sc head,small body}

\theoremstyle{fixed bf head,slanted body}
\newtheorem{thm}[res]{Theorem}          \newtheorem*{thm*}{Theorem}
\newtheorem{prp}[res]{Proposition}      \newtheorem*{prp*}{Proposition}
\newtheorem{cor}[res]{Corollary}        \newtheorem*{cor*}{Corollary}
\newtheorem{lem}[res]{Lemma}            \newtheorem*{lem*}{Lemma}

\theoremstyle{fixed bf head,upright body}
\newtheorem{dfn}[res]{Definition}       \newtheorem*{dfn*}{Definition}
     \newtheorem*{con*}{Construction}
      \newtheorem*{obs*}{Observation}
\newtheorem{rmk}[res]{Remark}           \newtheorem*{rmk*}{Remark}
          \newtheorem*{exa*}{Example}
         \newtheorem*{exe*}{Exercise}
            \newtheorem{stp*}{Setup}

\theoremstyle{numbered paragraph}
\newtheorem{ipg}[res]{}

\theoremstyle{unnumbered paragraph}
\newtheorem{ipg*}{}

\newlength{\thmlistleft}        %leftmargin
\newlength{\thmlistright}       %rightmargin
\newlength{\thmlistpartopsep}   %partopsep
\newlength{\thmlisttopsep}      %topsep
\newlength{\thmlistparsep}      %parsep
\newlength{\thmlistitemsep}     %itemsep

\newcounter{eqc} 
  {\end{list}}%

\newcounter{prt}
  {\end{list}}%

\newcounter{rqm}
  {\end{list}}%

  {\end{list}}%

  {\end{list}}%

\newcounter{exercise}
  {\end{list}}%

\newenvironment{prf*}[1][Proof]{%
  \begin{proof}[\bf #1]
    \setcounter{equation}{0}
    \renewcommand{\theequation}{\arabic{equation}}}
  {\end{proof}
}

\newcommand{\pgref}[1]{(\ref{#1})}

\renewcommand{\eqref}[1]{\pgref{eq:#1}}

\newcommand{\nt}[2][$^\diamondsuit$]{%
        \hspace{0pt}#1\marginpar{\tt\raggedleft #1 #2}}

\renewcommand{\aa}{\mathfrak{a}}
\newcommand{\bb}{\mathfrak{b}}
\newcommand{\m}{\mathfrak{m}}
\newcommand{\n}{\mathfrak{n}}
\newcommand{\p}{\mathfrak{p}}

\renewcommand{\le}{\leqslant}
\renewcommand{\ge}{\geqslant}

\newcommand{\Rm}{(R,\m)}
\newcommand{\Rmk}{(R,\m,k)}

\newcommand{\Rhat}{\widehat{R}}

\renewcommand{\Im}[1]{\nobreak{\operatorname{Im}#1}}
\newcommand{\Ker}[1]{\nobreak{\operatorname{Ker}#1}}
\newcommand{\Coker}[1]{\nobreak{\operatorname{Coker}#1}}

\newcommand{\AssR}{\operatorname{Ass}R}
\newcommand{\AsshR}{\operatorname{Assh}R}

\newcommand{\dimR}{\operatorname{dim}R}

\newcommand{\Spec}[1]{\operatorname{Spec}#1}

\newcommand{\ann}[2][R]{\operatorname{ann}_{#1}#2}
\newcommand{\Ann}[2][R]{\operatorname{Ann}_{#1}#2}

\newcommand{\Assh}[2][R]{\operatorname{Assh}_{#1}#2}

\renewcommand{\dim}[2][R]{\operatorname{dim}_{#1}#2}

\newcommand{\dpt}[2][R]{\operatorname{depth}_{#1}#2}

\newcommand{\hgt}[2][R]{\operatorname{ht}_{#1}#2}
\newcommand{\rnk}[2][k]{\operatorname{rank}_{#1}#2}

\newcommand{\id}[2][R]{\operatorname{id}_{#1}#2}

\newcommand{\Hom}[3][R]{\operatorname{Hom}_{#1}(#2,#3)}

\newcommand{\Ext}[4][R]{\operatorname{Ext}_{#1}^{#2}(#3,#4)}

\newcommand{\Tor}[4][R]{\operatorname{Tor}^{#1}_{#2}(#3,#4)}

\makeatletter
\def\@nobreak@#1{\mathchoice%
  {\nobreakdef@\displaystyle\f@size{#1}}%
  {\nobreakdef@\nobreakstyle\tf@size{\firstchoice@false #1}}%
  {\nobreakdef@\nobreakstyle\sf@size{\firstchoice@false #1}}%
  {\nobreakdef@\nobreakstyle\ssf@size{\firstchoice@false #1}}%
  \check@mathfonts}%
\def\nobreakdef@#1#2#3{\hbox{{%
                    \everymath{#1}%
                    \let\f@size#2\selectfont%
                    #3}}}%
\makeatother

\renewcommand{\nt}[2][]{} %Hides comments in tt

\numberwithin{equation}{res}
\newcommand{\xx}{\pmb{x}}
\newcommand{\cc}{\pmb{c}}

\renewcommand{\k}{\mathsf{k}}
\renewcommand{\id}{Id}

\newcommand{\kas}{\mathrm{KAS}}
\newcommand{\ho}{\mathrm{H}}

\renewcommand{\Rmk}{(R,\m,\k)}
\renewcommand{\rnk}[2][\k]{\operatorname{rank}_{#1}#2}

\newcommand{\im}{\operatorname{Im}}

\begin{document}

\sloppy \allowdisplaybreaks[4]

\title{Uniform Artin-Rees Bounds for syzygies}

\author[I. Aberbach]{Ian M.  Aberbach} 
\email{aberbachi@missouri.edu}
\address{Department of Mathematics, University of Missouri, Columbia, MO 65211, USA.}

\author [A. Hosry]{Aline Hosry}
\email{ahosry@ndu.edu.lb}
\address{Department of Mathematics and Statistics, Notre Dame University-Louaize,  Zouk Mosbeh, Lebanon}

\author[J. Striuli]{Janet Striuli}
\email{jstriuli@fairfield.edu}
\address{Department of Mathematics,  Fairfield University, Fairfield, CT~06824, USA.}

\thanks{The third author was supported by an AWM grant from May12th  to June 12th, 2009. The grant was used to visit the first author}

\date{\today}

%%\keywords{}

%%\subjclass[2000]{}

\begin{abstract}
  Let $\Rm$ be a local Noetherian ring, let $M$ be  
  a finitely generated $R$-module and let $(F_{\bullet},\partial_{\bullet})$
   be a free resolution of $M$. We find a uniform bound $h$ such that the Artin-Rees containment 
   $I^nF_i\cap \Im \partial_{i+1} \subseteq I^{n-h}\Im \partial_{i+1}$ holds 
   for all integers $i\ge d$, for all integers  $n\ge h$, and for all ideals $I$ of $R$.  In fact, we show that a considerably stronger statement holds.  The uniform bound $h$ holds for all ideals and all resolutions of $d$th syzygy modules. In order to prove our statements, we introduce the concept of Koszul annihilating sequences.
\end{abstract}

\maketitle

%%% INTRODUCTION
\section{Introduction}

Let $R$ be a Noetherian commutative ring.  To paraphrase Huneke in \cite{Hune-92}, the unsubtle finiteness condition that each ideal of $R$ is finitely generated often, in fact, implies subtle forms of finiteness in the ring, many of which are ``uniform.''  By ``uniform'' we mean, for instance, that one element may annihilate the homology of an entire class of complexes, or that an integer known to exist for a certain finite set of data (e.g., a finitely generated module and submodule, along with an ideal), may actually apply {\it uniformly}  to an infinite data set.  The Main Theorem, Theorem \ref{MainTheorem} below, is of the latter form, but in order to prove it, we have been led to the notion of what we call a Koszul annihilating sequence, which shows a form of finiteness of the former type.  

Let $R$ be a  Noetherian ring and let $M \subseteq N$ be  finitely
generated $R$-modules.  For an ideal
$I\subseteq R$, the classical Artin-Rees Lemma states that there
exists an integer $h$, depending on the ideal $I$ and the inclusion $M
\subseteq N$, such that for all integers $n \ge h$, the  equality 
\begin{equation}\label{ARequal}
I^nN \cap M =  I^{n-h}(I^hN \cap M)
\end{equation}
holds. A weaker form of this statement, which is often used in applications, is that for all integers $n \ge h$, the following inclusion holds:
\begin{equation}\label{ARcontain}
I^nN \cap M \subseteq  I^{n-h}M.
\end{equation}
The proof of the equality in equation \ref{ARequal} uses the fact that  the associated graded ring of $R$ with respect to $I$ and the associated graded modules of $M$ and $N$ with respect to $I$  are  Noetherian. With this approach,  the integer $h$ above, which is referred to as   the Artin-Rees number,  clearly depends on all of the data, that is, on $I$ and on $M \subseteq N$.

Eisenbud and Hochster, in \cite{EisHoc79}, first raised the question of whether or not there might be an integer $h$ as in equality \ref{ARequal} which, given $M \subseteq N$, works for all the maximal ideals in the ring $R$.  In this paper, the term ``uniform Artin-Rees'' was used for the first time, and an example in a non-excellent ring was given to  show that some condition on the ring is required in order to find a  uniform Artin-Rees number that works for all maximal ideals. In fact, for excellent rings, Duncan and O'Carroll in \cite{DO89} obtained a uniform  Artin-Rees number for which  equation \ref{ARequal}  holds  for every maximal ideal, answering Eisenbud and Hochster's question. Their work was preceeded by O'Carroll, who proved in \cite{OCar87} that an integer $h$ can be chosen to uniformly satisfy inclusion \ref{ARcontain} for all maximal ideals of an excellent ring.  
In  \cite{Hune-92}, Huneke showed that, 
  given the inclusion $M \subseteq N$, the integer $h$ in inclusion \ref{ARcontain} can be chosen
{\it independently} of the ideal $I$ under mild assumptions on the ring,
 for example when the ring is essentially of finite type over a local ring (by ``local'' we include the condition of being Noetherian).  Huneke's result gives a lot of information even in a local ring, whereas the O'Carroll and Duncan/O'Carroll results have content only in rings with infinitely many maximal ideals.  Our focus in this paper will remain on local rings, with an emphasis on uniform results over the set of all ideals.

 The impetus for our results is a ``uniform Artin-Rees'' problem   raised by
  Eisenbud and Huneke in \cite{EisHun-02}, concerning uniform
  bounds of Artin-Rees type on free resolutions.
 Let $M$ be a finitely generated $R$-module and let
  $(F_{\bullet}, \partial_{\bullet})$ be a
free resolution  of   $M$ by finitely generated free modules:
\[
\xymatrix{
\dots \ar[r] &F_{i+1} \ar[r]^{\partial_{i+1}} &F_{i} \ar[r] &\dots \ar[r]& F_1 \ar[r]^{\partial_1} & F_0 \ar[r] &M \ar[r]& 0.}
\]
  We define a module $M$  to be  {\it syzygetically Artin-Rees}
  with respect to a family of ideals $\mathcal{I}$, if there exists a uniform 
   integer $h$ such  that 
for every $n\ge h$,  for every $i\ge 0$, and for every $I \in \mathcal{I}$,
\begin{eqnarray}\label{syzyg}
I^nF_i\cap \Im{ \partial_{i+1}}\subseteq I^{n-h}\Im{ \partial_{i+1}}.
\end{eqnarray}
This definition does not depend on the free  resolution (see for example Lemma 2.1 in \cite{Stri-06b}).  Eisenbud and Huneke raise the question (Question B in \cite{EisHun-02}) of whether or not, given a local ring $(R,\m)$ and an ideal $I \subseteq R$, every module is syzygetically Artin-Rees with respect to $\mathcal{I} = \{I\}$.

 In \cite{EisHun-02}, the authors show that given a local ring $(R,\m)$, a module $M$ is syzygetically Artin-Rees with respect to an ideal $I \subseteq R$, if $M$ has finite projective dimension and constant rank on the punctured spectrum.  In \cite{Stri-06b}, the third author of this paper observed that much stronger statements could be made if one looks at high enough syzygies.
 She proved that every second syzygy module in a local ring of dimension at most two has a certain uniform annihilation property with respect to the family of $\m$-primary
  ideals (see Theorem 5.4 of \cite{Stri-06b}).  Although it is never formally stated in her paper, one can deduce from her results that in a local ring of dimension two, every finitely generated second syzygy is syzygetically Artin-Rees with respect to the set of $\m$-primary ideals, and that the integer $h$ in inclusion \ref{syzyg} can be chosen independently of the module $M$, i.e., we have a {\it double uniform} Artin-Rees theorem. This leads us to an extended definition: We say that a family of modules $\mathcal{M}$ is syzygetically Artin-Rees with respect to a family of ideals $\mathcal{I}$ if there exists an $h$ such that inclusion \ref{syzyg} holds for every module of $\mathcal{M}$ and every ideal of $\mathcal{I}$.
It is in this form that we prove our main theorem.
  
\begin{ipg}{\bf{Main Theorem.}}\label{MainTheorem} 
Let $\Rm$ be a local Noetherian ring of dimension $d$. Then, the family of finitely generated $d$th syzygy modules is
  syzygetically Artin-Rees with respect to the family of all ideals. 
 \end{ipg}	
	
  It quickly becomes apparent, when considering this problem over infinite sets of ideals and modules, that no uniform statement is possible when we consider the beginning of a resolution, even in very nice rings.  Suppose that $(R,\m)$ is Cohen-Macaulay (or even regular) and $x_1,\ldots, x_d$ is a system of parameters.  The Koszul complex $K_\bullet(x_1^t,\ldots, x_d^t;R)$ is a free resolution of the module $R/(x_1^t,\ldots, x_d^t)$.  This complex will have Artin-Rees number at least $t$ for the ideal $(x_1,\ldots, x_d)$.  But there {\it is} a uniform behavior for all such complexes past $d$, since all higher syzygies are zero.  
	
%	We will show, in fact, that {\it all} local rings behave well for all high enough syzygies.  We show (see Theorem~\ref{mainthm})
 
 The syzygetically Artin-Rees property is intertwined with finding uniform annihilators for homology modules.
In fact, given a free resolution $(F_{\bullet}, \partial_{\bullet})$ of an $R$-module $M$, there is an isomorphism 
\begin{eqnarray}\label{torUAR}
\frac{I^nF_{i-1} \cap \Im(\partial_{i})}{I^n\Im \partial_{i}} \cong \Tor{i}{R/I^n}{M}, \quad \text{for every $i \ge 1$}.
\end{eqnarray}

Our work to find {\it uniform} annihilator elements for the homology module displayed above, is inspired by the ideas contained  in the monograph \cite{HoHu-93}, which also relate to important homological conjectures. In this monograph, the authors show 
that elements that annihilate higher homology modules  of Koszul complexes on  sequences that are  part of a system of parameters, can often, after taking a power, annihilate higher homology modules of a much wider class of complexes.   Huneke's proof of his uniform Artin-Rees and uniform Brian\c con-Skoda theorems also relies heavily on such elements.
 
 In order to prove our Main Theorem, given a local ring $(R,\m)$, we need a {\it sequence }of elements which annihilate higher Koszul homology modules of parameters.  As we need this sequence to itself form a system of parameters, we are forced to choose elements that annihilate higher Koszul homology of part of system of parameters of appropriate length. 
This leads us to introduce the notion of {\it Koszul annihilating sequence} (KAS), whose precise definition, Definition~\ref{KASdef}, is given in Section 2, together with the proof of its  existence,  Theorem~\ref{KASexistence}.  We believe that KAS sequences may offer a significant tool in problems where controlling higher homology modules for classes of free complexes is important.
 
Section 3 is devoted to showing how Koszul annihilating sequences uniformly annihilate the homology modules of a much larger class of complexes. %, which were also introduces in \cite{HoHu-93}. 

Finally, in Section 4, after proving several lemmas,  we give the proof of Theorem \ref{mainthm} of which the Main Theorem is a corollary.

% In order to show that KAS's exist, we rely on a variant of Theorem 2.16 of Hochster and Huneke's paper \cite{HoHu-93} (cited in this paper as result \ref{annkoszul}).  Their theorem allows one to use the knowledge that if $R$ has an element $c$ which annihilates certain local cohomology modules of a finitely generated module $M$, then there is a power of $c$ which annihilates certain higher Koszul homology modules on $M$ for elements satisfying a height condition on the module $M$.  We give a variant of this result, Theorem~\ref{altannkoszul}, where we relax this height condition on $M$, at the cost of stronger assumptions on the ring.  This cost, however, is not much of a cost because we can apply the theorem after completing, and we can then view (the completion of) $M$ as a module over a Gorenstein ring.

%%%%%%%%%%%%%%%%%%%%%%%%%%%%%%%%%%%%%%%%%%%%%%
%%%%%%%%%%%%%%%%%%%%%%%%%%%%%%%%%%%%%%%%%%%%%%%%
%%%%%%%%%%%%%%%%%%%%%%%%%%%%%%%%%%%%%%%%%%%%%%%%

%%%%%%%%%%%%%%%%%%%%%%
%%%%%%%%%%%%%%%%%%%%%%
%SECTION 1: KAS
%%%%%%%%%%%%%%%%%%%%%%
%%%%%%%%%%%%%%%%%%%%%%

\section{Koszul annihilating sequences}
Throughout this section,  $R$  will denote  a Noetherian ring of dimension $d$.
Let  $\xx=x_1, \dots, x_n $ be a sequence of elements in $R$, and let  $M$ be an $R$-module. We denote the Koszul complex on the sequence $\xx$  and the module $M$ by $K_{\bullet}(x_1,\dots, x_n; M)$, and we denote its $i$th homology by $\ho_i(x_1,\dots, x_n; M)$.
 
 This section is devoted to showing, given a local ring $(R,\m)$,  the existence of a system of parameters, which we call a {\it Koszul annihilating sequence} (or KAS),  whose elements uniformly annihilate higher homology modules of   the Koszul complex  on sequences that are part of a system of parameters against all  modules that are high syzygies.  The precise definition of KAS is given in Definition~\ref{KASdef}, and the existence of such a sequence for the class of $d$th szyzygies is the content of Theorem~\ref{KASexistence}.
 
As mentioned in the introduction, our work is inspired by the monograph of Hochster and Huneke. The first result we need to modify is the following which is
 Theorem (2.16) of \cite{ HoHu-93}.  For a finitely generated module $M$, the height of $I$ on $M$, $\hgt[M]I$, is the height of the ideal $I$ in the ring $R/\Ann[R]M$.
\begin{ipg}\label{annkoszul} 
Let $R$ be a Noetherian ring, not necessarily local. 
Let $x_1, \dots, x_n$ be elements in the ring $R$
 and suppose that $M$ is a finitely generated 
 $R$-module  of dimension $d'$ such that the inequality $\hgt[M](x_1, \dots, x_i) \ge i$ is satisfied  for all  $i=1, \dots, n$. 
 Let $c$ be an element of $R$ such that, for $i=0, \dots, n-1$, the equality 
  $c \ho^{i}_{\p R_{\p}}(M_{\p})=0 $ holds for any prime ideal $\p$ containing the elements
   $x_1, \dots, x_{i+1}$. Then the equality  $c^{E(d', n, t)}\ho_{n-t}(x_1, \dots, x_n; M)=0$ holds 
    for $0 \le t \le n-1$, where, for integers
   $\delta \ge \nu > \tau \ge 0$, the function $E(\delta, \nu, \tau)$ is defined recursively
    as follows:
   $$
 \begin{cases}
 E(\delta, \nu, 0)=\delta-\nu+1\\
  E(\delta, \nu, \tau)=\delta+ (\delta+2)E(\delta, \nu-1, \tau-1), \quad \tau\ge 1.
  \end{cases}
  $$

\end{ipg}

We provide here an alternative version of result~\ref{annkoszul},
 with essentially the same proof as given in \cite{HoHu-93}.  
 With stronger hypotheses on $R$, we may look only at the 
 height of the sequence on $R$ (not on the module $M$), 
 and we may drop the dimension assumption on $M$. This
  is valuable for our applications to syzygies
  as we do not always know that syzygies have 
  the full dimension of the ring. On the other hand, the hypotheses on the ring will be satisfied as we will apply our theorem to Gorenstein rings.

In our  case, we  define a new function $E_1(\delta, \nu, \tau)$ recursively by
$$
 \begin{cases}
 E_1(\delta, \nu,0)= \delta-\nu+1\\
  E_1(\delta, \nu, \tau)=\delta+ (\delta+2)E_1(\delta-1,\nu-1, \tau-1), \quad \tau\ge 1.
  \end{cases}
  $$

\begin{thm}\label{altannkoszul}
Let $R$ be a Noetherian ring of dimension $d$ which
 is catenary and  locally equidimensional 
 at each maximal ideal. 
  Let $\xx = x_1,\ldots, x_n $ be a sequence of elements of $R$ for which the inequality
   $\hgt (x_1,\ldots, x_i)R \ge i$ holds for all integers $i$ with $1 \le i \le n$. 
    Let $M$ be a finitely generated $R$-module, and 
    let $c \in R$ be an element such that the equality
     $c\ho^i_{\p R_\p}(M_\p)=0$ holds
      for any prime ideal $\p$ containing  the sequence 
      $x_1,\ldots, x_{i+1}$ and  for $0 \le i \le n-1$.
      
       Then, the equality 
$c^{E_1(d,n,t)}\ho_{n-t}(\xx; M) = 0$ holds for $0 \le t \le n-1$.
\end{thm}

\begin{proof}
We use induction on $d$, $n$, and $t$.  
The property of being equidimensional at maximal ideals  localizes and 
the statement  is a local one on $R$ at maximal ideals. Moreover, if
 $\m$ is a maximal ideal of $R$ then $\hgt \m \le d$. 
  Since for $d' \le d$ we have $E_1(d', n, t) \le E_1(d, n, t)$, without
 loss of generality we may assume that $(R,\m)$ is a local Noetherian catenary and
    equidimensional ring of dimension $d$.

If $t=0,$ we shall show that $\ho_n(\xx;M)$ is annihilated  by $c^{d-n+1}$. 
 We can assume that $n \ge 1$. 
If $n > d$, then $(\xx)R = R$ and all Koszul homologies vanish.
 If $d =n$, then $(\xx)$ generates an $\m$-primary ideal and 
 $\ho_n(\xx;M) \subseteq \ho^0_\m (M)$, which  is annihilated by $c$  by hypothesis. 
  We may thus assume that $n < d$.  Consider $c^{d-n}\ho_n(\xx;M)$.  
  By the induction hypothesis on $d$, this module vanishes on the punctured spectrum.  
  The module is also a submodule of $M$, 
  and hence $c^{d-n}\ho_n(\xx;M) \subseteq \ho^0_\m(M)$.  
  By hypothesis, $c$  is in the annihilator of the module $c^{d-n}\ho_n(\xx;M)$.
 
 We now assume that $t>0$. 
 Let $N = \ann[M](x_1) = \ho_1(x_1;M)$. 
  We may apply the case $t=0$ and $n =1$
  to see that $c^d$ is in the annihilator of the module $N$.  
  Consider the two exact sequences 
  \begin{equation}\label{X2}
  0 \to N \to M \to M' \to 0, \text{ and}
  \end{equation}
   \begin{equation}\label{X3}
   0\to  M' \to M \to M/x_1M \to 0, \; \text{ where} \;M' = x_1M \cong M/ N.
   \end{equation} 
 
  Let $r$ be an integer with $0 \le r \le n-1$ and $\p$ a prime ideal containing $x_1,\ldots, x_{r+1}$.
  
  The  short exact sequence \ref{X2}  yields the following long exact sequence:
\begin{equation}\label{XO}
\cdots \to \ho^i_{\p R_\p}(M_\p) \to H^i_{\p R_\p}(M'_\p) \to \ho^{i+1}_{\p R_\p}(N_\p) \to \cdots.
\end{equation}
Since the equality $c^dN=0$ holds, then  $c^d$ is in 
the annihilator of the third module displayed in
 \ref{XO}  for all $i \ge 0$;  by hypothesis 
 $c$ annihilates  the first module displayed in
  \ref{XO} for  all $i \le r$.  Thus, for $i \le r$, $c^{d+1}H^i_{\p R_\p}(M'_\p)=0$.

In particular, $c^{d+1}$ kills the third term displayed below for all $i<r$ 
\begin{equation}\label{X1}
\cdots \to \ho^i_{\p R_\p}(M_\p) \to \ho^i_{\p R_\p}((M/x_1M)_{\p}) \to \ho^{i+1}_{\p R_\p}(M'_\p) \to \cdots,
\end{equation}
where the long exact sequence comes from the  short exact sequence \ref{X3}.
 Hence, for all $i < r$, $c^{d+2}$ kills the middle term in \ref{X1},
  since by hypothesis $c \ho^i_{\p R_\p}(M_\p)=0$  for $i\le r$.

Because $R$ is equidimensional and catenary of dimension $d$, the ring
 $R/x_1R$ is equidimensional and catenary of dimension $d-1$, 
 and $\hgt\left((x_2,\ldots, x_{i+1})(R/x_1R)\right)\ge i$.  
 We apply the induction hypothesis 
 with $M/x_1M$ replacing $M$, $R/x_1R$ replacing
  $R$, $x_2,\ldots, x_n$ replacing the original sequence,
   $c^{d+2}$ replacing $c$, and $d-1$, $n-1$, $t-1$
    replacing $d$, $n$, and $t$ respectively
  to conclude that $c^{(d+2)E_1(d-1, n-1, t-1)}$ kills $\ho_{n-t}(x_2,\ldots, x_n; M/x_1M)$.

Finally, from a spectral sequence there is a long exact sequence
\begin{equation*}
\cdots \to \ho_{i-1}(x_2,\ldots, x_n; N) \to \ho_i(\xx;M) \to \ho_i(x_2,\ldots, x_n; M/x_1M) \to \cdots.
\end{equation*}
Let $i = n-t$.  Since $c^d$ annihilates  the first displayed term and $c^{(d+2)E_1(d-1, n-1, t-1)}$
 annihilates the third term, we obtain that $c^{E_1(d, n, t)} \ho_i(\xx;M)=0$, as required.
\end{proof}

We are now ready to give the definition of a Koszul annihilating  sequence.

\begin{dfn}\label{KASdef}
Let $(R,\m)$ be a local ring of dimension $d$ and let $M$ be a finitely
 generated $R$-module.  Given an integer $i$ such that $i \le d$, a sequence of elements, 
 $c_i, c_{i+1}, \ldots, c_d\in \m $ is a {\it Koszul annihilating sequence}  
 ($\kas$) for $M$ (of length $d-i+1$), if 
  \begin{enumerate}
 \item  the elements $c_i, \ldots, c_d$ are part of a system of parameters for $R$, and
 \item for all integers $k$ such that 
 $d\ge k \ge i-1$, if $x_1,\ldots, x_k, c_{k+1}, \ldots, c_d$
  are a  system of parameters for $R$,
          then the equality
          $$c_v H_n (x_1,\ldots, x_k, c^t_{k+1}, \ldots, c^t_j; M) = 0$$
          holds for all $v \ge j \ge k \ge 1$ and for all $n, t\ge 1$.
 \end{enumerate} 

If $\mathcal{M}$ is a family of finitely generated $R$-modules, and
 $c_i,\ldots, c_d$ is a $\kas$ for all $M \in \mathcal{M}$, then we will say that
  the sequence is a $\kas$ for $\mathcal M$.
\end{dfn}

\begin{rmk}  It is clear from the definition that if we modify 
a $\kas$ by replacing each element by  a power of 
itself, we still have a $\kas$.

Moreover a $\kas$ of length $\dimR$ is a system of parameters for $R$.
\end{rmk}

Theorem \ref{KASexistence} below shows that  for any local ring, it is always possible to find a
 $\kas$ of length $d$  for the class of modules which are $d$th syzygies.
  Before proving the theorem, we introduce some ideals that will be used in the proof.

\begin{ipg} \label{setup}
Let $(R,\m)$ be a local ring.  For an integer $i$ such that $0 \le i < d$, 
 set $\aa_i = \aa_i(R) = \Ann{H^i_\m(R)}$ and $\bb_i = \bb_i(R) = \aa_0(R) 
\cdots \aa_i(R)$.  It is a simple induction to see that if $M$ is a $j$th 
syzygy then for $0 \le i <j$, $\bb_i \subseteq \Ann{H^i_\m(M)}$. 
 In particular, for every $d$th syzygy, and for every $0 \le i <d$,
  $\bb_i \subseteq \Ann{H^i_\m(M)}$.   We also note that 
  $\bb_0 \supseteq \bb_1 \supseteq \cdots \supseteq \bb_{d-1}$ and $\dim[]R/\bb_i \le i$.
\end{ipg}

\begin{thm} \label{KASexistence} 
Let $(R,\m)$ be a local ring of dimension $d$.  
Then, there exists a $\kas$ sequence, $c_1,\ldots, c_d$,
 for the family of finitely generated modules which are $d$th syzygies over $R$.

Moreover,  the $\kas$ sequence $c_1, \dots, c_d$ can be chosen such that 
 the following holds
 \[ \dim[]R/(0:_R(0:_Rc_{d-i})) \le d-i-1,
 \]
  for all $1 \le i \le  d-1$. 
\end{thm}
\begin{proof}

We begin by choosing $c'_d \in \bb_{d-1} - \bigcup_{\p \in \AsshR} \p$, 
where $\AsshR$ denotes the set of associated prime 
ideals  $\p$ of maximal dimension, i.e.,   $\dim[]R/\p =\dimR$. 
This choice is possible, since $\dim[]R/\bb_{d-1} \le d-1$  by \ref{setup},
 and therefore, by prime avoidance,
 $\bb_{d-1}$ is not contained in 
 any minimal prime of $R$ of maximal dimension.  
 
 Inductively, having chosen $c'_d, \ldots, c'_{i+1}$,  we pick 
\begin{equation*}
c'_i \in \bb_{i-1} -\bigcup\left(\{\p \mid {\p \in \Assh{R/(c'_{i+1},\ldots, c'_d)}}\} \cup\{\p \mid \p \in \AssR, \dim R/\p \ge i\}\right).
\end{equation*}
We can avoid all the above primes in question
 since $\dim[]R/\bb_{i-1} \le i-1$ by \ref{setup}. 

 For $1 \le i \le d$, set $c_i = (c'_i)^{E_1(d,d,d-1)}$.

The elements $c_d, \ldots, c_1$  form a system of parameters
 in $R$, and hence in $\Rhat$.  Let $(S, \n)$ be a Gorenstein 
 ring of dimension $d$ mapping onto $\Rhat$, say $\Rhat = S/I$.  
 Again, using prime avoidance, we may lift $c_1,\ldots, c_d$ to 
 elements in $S$ which form a system of parameters in $S$. 
  Abusing terminology, we will continue to call the elements
   $c_1,\ldots, c_d$.  For all $\Rhat$ 
   modules,  the action of the lifted elements is the same as the action of the $c_i$'s.

Let $M$ be a finitely generated $d$th $\Rhat$-syzygy. 
 Suppose that $x_1, \ldots, x_k \in S$ are elements such 
 that $x_1,\ldots, x_k, c_{k+1}, \ldots, c_d$ are a system of parameters in $S$.   
  We want to show that, given  integers $v \ge j \ge k$ and $t, n \ge 1$, the equality
$c_v H_n(x_1,\ldots, x_k, c_{k+1}^t, \ldots, c_j^t; M) = 0$ holds. 

 By Theorem~\ref{altannkoszul}, it suffices to show that the equality 
 $c'_vH^i_{\p S_{\p}}(M_{\p})=0$ holds  whenever the index $i$ satisfies $0 \le i \le  j-1$ and the prime ideal  $\p \in \Spec(S)$ is such that 
 $(y_1, \ldots, y_{i+1}) \subseteq \p$, where $y_m = x_m$ for $1 \le m \le k$ and $y_m = c_m^t$ otherwise.

Fix an integer  $i \in \{0, \dots, j-1\}$, by local duality the
 cohomology module $H^i_{\p S_\p}(M_\p)$ is dual to $\Ext [S] {\hgt [] \p - i} {M}{S}_\p$; further
the module $ \Ext [S] {\hgt [] \p - i} M S$ 
is dual to the cohomology module $H^{d- \hgt [] \p + i}_\n (M)$.  
 If $v \ge d - \hgt[] \p +i +1$ then, using \ref{setup}, 
  the equality $c_v H^{d- \hgt [] \p+ i}_\n (M)=0$ holds,  
  as the element  $c_v \in \bb_{v-1}\subseteq \bb_{d- \hgt [] \p + i}$.
 Thus we may assume that $v < d - \hgt[]\p + i + 1$,
  or, equivalently, that $\hgt[]\p < d-v+i+1$.  
  We claim that the ideal $\p$ does not contain all the elements
   $c_d, \ldots, c_{d - (\hgt[]\p -i -1)}$. Indeed if $\p$ would
    contain these elements as well as $y_1,\ldots, y_{i+1}$, 
    then $\p$ contains at least $\hgt[]\p+1$ elements from a system
     of parameters for the ring $S$, contradicting the Gorenstein assumption on $S$.
  
Therefore $\p$ does not contain all the elements
$c_d, \ldots, c_{d - (\hgt[]\p -i -1)}$ so that
 the module $ \Ext [S] {\hgt [] \p - i}{ M }{S}_\p$ vanishes, as  its annihilator contains the elements $c_d, \ldots, c_{d - (\hgt[]\p -i -1)}$.
% The condition on the height of $P$ 
% suffices to assure us that the union of this set of elements is part
%  of a system of parameters for $S$, but the union contains
%   $\hgt[]P + 1$ elements, a contradiction.  
   Hence the hypotheses of Theorem~\ref{altannkoszul} hold.  In particular, this shows that $c_1,\ldots, c_d$ is a KAS for $\Rhat$.

Note that the elements $c_1,\ldots, c_d$ live in $R$.  Because  $\Rhat$ is faithfully flat over $R$, if $M$ is a finitely generated $d$th syzygy over $R$ then $M \otimes_R \Rhat$ is a $d$th syzygy over $\Rhat$.  Also, $-\otimes_R \Rhat$ commutes with taking Koszul homology so the vanishing condition over $\Rhat$ implies the same vanishing over $R$.  Thus $c_1,\ldots, c_d$ is a KAS over $R$.

Since we picked each $c_i$ to avoid appropriate associated primes of $R$, 
we also have that $\dim[]R/(0:_R(0:_R c_{d-i})) \le d-i-1$ for each $i$.
 \end{proof}
$\kas$ sequences  annihilate homology modules of Koszul complexes on every sequence $\xx$ (of appropriate length) that  is part of a system of parameters. However, in Lemma \ref{celimination}, we will need to carefully choose  sequences $\xx$ with  properties that are captured by the following definition.

\begin{dfn} Let $\cc = c_1,\ldots, c_d$ be a KAS for
  a family of modules $\mathcal{M}$. 
 A system of parameters $\xx = x_1,\ldots, x_d$ is said to be
  {\it well-suited to} $\cc$ if for all $1 \le i \le j \le d$,  any
   subset of cardinality $d-(j-i+1)$ of the sequence $\xx$ together with
  the sequence $ c_i, \ldots, c_{j}$
   is   a system of parameters for $R$.
\end{dfn}

Using this definition (with $j=d$ and $j=d-1$ respectively), the two equalities in the  following corollary are particular cases of Theorem \ref{KASexistence}.
\begin{cor}\label{KASUSED} 
Let $\cc = c_1, \dots, c_d$ be a $\kas$ sequence for the family of modules which are
 $d$th syzygies and let
$x_1, \dots, x_d$ be a well-suited sequence to $\cc$. Then
          $$c_v H_t (x_1,\ldots, x_k, c_{k+1}, \ldots, c_{\ell}; M) = 0,$$
           $$c_v H_t (x_1,\ldots, x_k, c_{k}, \ldots, c_{\ell-1}; M) = 0,$$
 for all  integers
$v \ge  \ell \ge k \ge 1$, for all $t\ge 1$,  and for all $R$-modules $M$ that are $d$th syzygies.
\end{cor}

\section{Uniform annihilators  of homology}
%%%%%%%%%%%%%%%%%%%%%%%%%%
%%%FROM KOSZUL TO HOMOLOGY OF COMPLEX
%%%%%%%%%%%%%%%%%%%%%%%%%%%%
 
As shown in \cite{HoHu-93}, elements that annihilate the Koszul homology are also in the annihilator of the homology modules of a wider class of complexes. In this section we look at a class of complexes described, for instance,  in  Definition (3.16) of \cite{HoHu-93}
and we study their relations with  $\kas $ sequences.

 Let $(G_{\bullet}, \partial _{\bullet})$ be
  a complex of finitely generated free modules:
\[
\xymatrix{ 0 \ar[r] &G_n \ar[r]^{\partial_n} &G_{n-1} \ar[r] &\dots
  \ar[r]& G_1 \ar[r]^{\partial_1} & G_0 \ar[r] &0}
\]
where $n$ is the length of the complex.
 For a positive integer $r$, denote by
$I_r(\partial_i)$ the ideal generated by the $r\times r$ minors of a
matrix that represents the map $\partial_i$ after a choice of basis
 (the ideal is independent of the choice). 
 The rank of the homomorphism $\partial_i$ is given by
\[
\rnk[]{\partial_i}=\max\{r \mid  I_r(\partial_i) \neq 0\}.
\]

Denote by $I(\partial_i)$ the ideal
$I_{\rnk[]{\partial_i}}(\partial_i)$.

The complex $(G_{\bullet},\partial_{\bullet})$ satisfies the {\it standard
conditions on rank} if the following equality holds for $i=1, \dots, n$:
\[
\rnk [R]{G_i}=\rnk[]{\partial_i }+ \rnk[]{\partial_{i+1}}.
\]

\begin{ipg}\label{eisbuc}
The standard conditions on rank play a very important role on the exactness 
of a complex.  Buchsbaum and Eisenbud show in  \cite{BucEis-73}  that
a complex of free modules
\[
\xymatrix{ 0 \ar[r] &G_n \ar[r]^{\partial_n} &G_{n-1} \ar[r] &\dots
  \ar[r]& G_1 \ar[r]^{\partial_1} & G_0 \ar[r] &0}
\]
is exact if and only if it satisfies the standard conditions on rank and $\dpt I(\partial_i) \ge i$ for all integers $i \ge 1$.
\end{ipg}
The proof of  \ref{eisbuc} uses several lemmas. We report two of them for easy reference. One could consult  Proposition 1.4.10 of \cite{bruns-herzog} or refer to the original paper \cite{BucEis-73} for a proof.

\begin{ipg}\label{baseOfStandard}
Let $\partial: F_1\to  F_0$ be a homomorphism between two free $R$-modules. 
Denote by $M$ the cokernel of $\partial$. Then $I(\partial)=R$ if and only if $M$ is a free module of rank equal to  $\rnk[] F_0 -\rnk[](\partial)$.
\end{ipg} 

\begin{ipg}\label{TT}

Let $0 \to G_n \to \dots \to G_0$ be a complex of free 
modules that satisfies the standard conditions on rank.
 Assume that $I(\partial_i)=R$ for all
  $i=1, \dots, n$ then the complex is split exact.  
 
 \end{ipg}

Our first proposition will show how  a $\kas $ sequence annihilates the homology of complexes which are the tensor product of high syzygies and  certain complexes satisfying the standard conditions on rank. We will use the following result, which is  Proposition 3.1 of \cite{HoHu-93}. 
 The statement in the original publication has a typo, so we restate it here.
\begin{ipg}\label{foundation-homology}
Let $R$ be an arbitrary commutative ring and let 
$M_\bullet$ be a complex 
$$0 \to M_n \to \cdots \to M_1 \to M_0 \to 0$$ 
of arbitrary $R$-modules.  Let $\xx = x_1,\ldots, x_n $ be a sequence of elements in $R$, and 
let $d, d_0, d_1,\ldots, d_{n-2}$ be elements of $R$ such that 
\begin{enumerate}
\item $d_i$ kills $H_{n-i}(M_\bullet)$, for $0 \le i \le n-2$ 
(this is where the correction is), and
\item $d$ kills $H_{n-j}(x_1,\ldots, x_n; M_{j+1})$, for $1 \le j \le n-1$.
\end{enumerate}
Then $D = (d_0d_1\ldots d_{n-2})d^n$ kills $\Hom[R] {R/(x_1,\ldots, x_n)}{H_1(M_\bullet)}$.
\end{ipg}

\begin{prp}\label{KAScomplex}
Let $(R,\m)$ be a local ring of dimension $d$ and let
 $c_1,\ldots, c_d$ be a KAS for the family of modules which are $d$th syzygies.
  Then there exists an integer $t \ge 0$ such that
 $$c_{n+j}^tH_i(G_\bullet \otimes M)=0,$$
 for all integers $i$ and $j$ such that  $i \ge  1$, $0\le j \le d-n$,  for every $d$th syzygy $M$, and 
for every complex $(G_\bullet,\partial_\bullet)$ of length $n\le d$ 
of finitely generated free $R$-modules with the following properties:
\begin{enumerate}
\item  $G_\bullet$ satisfies the standard  conditions on rank,
\item  For $1 \le i \le n$,  the ideal $I(\partial_i) + (c_{i+1},\ldots, c_d)$ is $\m$-primary.
\end{enumerate}
\end{prp}

\begin{proof}
Denote by $b_i$ the rank of the free module $G_i$.
We use reverse induction on the homological
 degree of the complex $G_{\bullet}$.

At homological degree $n$,  the complex looks like $0 \to G_n \to
  G_{n-1}$, with $\partial_n: G_n \to G_{n-1}$.  
  By hypothesis $I(\partial_n) +(c_{n+1},\ldots, c_d)$ is $\m$-primary.   
   Choose elements $y_1,\ldots, y_n \in I(\partial_n)$
  such that the sequence  $y_1,\ldots, y_n, c_{n+1},\ldots, c_d$ forms a system of parameters.  
  Then for each $1 \le i \le n$, we obtain ${I(\partial_n)}_{y_i}=R_{y_i}$. 
  The content of \ref{baseOfStandard} implies that the cokernel of $(\partial_n)_{y_i}$ is a free module of rank  equal to
  $b_{n-1}-\rnk[](\partial_n)_{y_i}=b_{n-1}-b_n$, where the last equality holds since the complex $G_{\bullet}$ satisfies the standard conditions on rank. In particular, the sequence 
$$0 \to (G_n)_{y_i} \to (G_{n-1})_{y_i} \to  \Coker (\partial_n)_{y_i} \to 0$$
 is a split exact sequence.
  
  Let $M$ be a $d$th syzygy and let $\mathbf{z}\in \ho_n(G_{\bullet}\otimes M)\subseteq  G_n\otimes M$.   
   Then $(G_n\otimes M)_{y_i} \to (G_{n-1}\otimes M)_{y_i}$ is injective for each $i$ such that $0\le i\le n$, so there is an integer $s$ such that   $y_i^{s}\mathbf{z}=0$
    for each $i$ such that $1\le i \le n$,  
    and therefore $\mathbf{z}\in \ho_n(y_1^{s}, \ldots, y_n^s; G_n\otimes M)$. 
By Corollary~\ref{KASUSED}, each of $c_{n}, \ldots, c_d$ kills this homology, and so each kills $\mathbf{z}$.

Suppose that the conclusion holds for all complexes of length $n$,
 at each homological degree greater than $m$ (with $1 \le m < n$).
  In particular,  there exists an integer  $k$ such that  $c_n^k,\ldots, c_d^k$ 
 annihilate $\ho_i(G_{\bullet} \otimes M)$ for all $i >m$. 
 We need to show that there is a power of the elements $c_n, \dots, c_d$ that annihilates $\ho_m(G_{\bullet}\otimes M)$.

Pick $y_1,\ldots, y_m \in I(\partial_m)$ such that $y_1,\ldots, y_m, c_{m+1},\ldots, c_d$ are a system of parameters.  For each $1 \le i \le m$, the localized subcomplex $0 \to (G_n)_{y_i} \to \cdots \to (G_m)_{y_i}\to (G_{m-1})_{y_i}$ is split exact since the complex satisfies the hypothesis of \ref{TT}.  Hence, given a $d$th syzygy $M$, there is a power $s$ such that $(y_1^s,\ldots, y_m^s)$ kills  the homology $H_m(G_\bullet \otimes M)$. This implies that there is a homomorphism that sends the identity in $R/(y_1^s,\ldots, y_m^s)$ to any element
 $\mathbf{z} \in H_m(G_\bullet \otimes M)$.  
 As the elements $c_i$'s are part of a KAS, we 
 may now apply the content of \ref{foundation-homology} to the complex 
$(0 \to G_n \to \cdots \to G_{m-1} \to 0) \otimes M$ 
with each $d_0 = \cdots = d_{n-m-1} = c_j^k$ 
and $d = c_j$ for $j\ge n$.  Thus $c_j^{(n-m)k + n-m+1}$ 
kills the desired homology. 
\end{proof}

%%%%%%%%%%%%%%
%%%OUR COMPLEXES
%%%%%%%%%%%%%
We now look at the particular complex we will use.

\begin{rmk}\label{EaNo}
Let $ \xx=x_1, \dots, x_h$ be a sequence elements of a ring $R$. 
%such that $\hgt[R](x_1, \dots, x_i) \ge i$ for all $i=1, \dots, h$.
Consider the $n \times (n+h-1)$ matrix
\[ 
\mathcal{B}=\left( \begin{array}{cccccccc}
x_1  & x_2  & \dots  &  x_h    & 0            & \dots &\dots & 0 \\
0      &  x_1 & x_2     & \dots   & x_{h}&  0& \dots &0 \\
\dots &        &             &              &               &         &          &   \\
0      &   0     &  \dots   &  \dots  &               &         &           & x_h\\
 \end{array} \right).
 \] 
The $n \times n$ minors of $\mathcal B$ are the generators 
of the ideal $J^n$, where $J=(x_1, \dots, x_h)$. 
In \cite{EaNo62} the authors construct a complex 
$(B^{J, n}_{\bullet}, \partial_{\bullet})$ 
such that $\ho_0(B^{J, n}_{\bullet})=R/J^n$.
Theorem 2 of \cite{EaNo62} says that  if  $\xx$
 is a regular sequence, then the complex
  $(B^{J, n}_{\bullet}, \partial_{\bullet})$ 
  is a free resolution of $R/J^n$ of length $h$.
   Moreover, $J \subseteq \sqrt{I(\partial_i)}$ for all $i=1, \dots, h$.
% 
% Let $J$ be a monomial ideal in $x_1, \dots, x_h$ and let  $\mathbf{x}$ be a monomial in $x_1, \dots x_h$ not in $J$.  In Corollary 3.4 of \cite{--} the authors prove that  there exists a complex $G_{\bullet}$ of length $h$ that satisfies the star dad conditions on rank and strong conditions on height such that 
%   
\end{rmk}
Using the particular complex above, we are finally able to list in the following proposition all the results that are used in the next section.

\begin{prp}\label{fromAGT}
 Let $\Rmk$ be a local ring of  dimension $d$, let $\mathcal{M}$ be the family of all finitely generated $R$-modules that are $d$th syzygies,and
 let $\cc = c_1, \dots, c_d$ be a $\kas$ sequence for $\mathcal{M}$.  
 Let $x_1, \dots, x_d$ be 
 a system of parameters well-suited to $\cc$ and denote by $I_j$ the ideal $(x_1, \dots x_j)$.
 Let $t$ be the integer given by Proposition~\ref{KAScomplex} for $\cc$. 
  For every $i$ and $j$ such that $1 \le j \le i \le n$ and   for any positive integers $t_{j+1}, \ldots, t_i$ the following hold:
 \begin{enumerate}
 
\item  the complex 
$(B^{I_j, n}_{\bullet}, \partial_{\bullet}) \otimes K_\bullet(c_{j+1}^{t_{j+1}}, \ldots, c_i^{t_i}; R)$ 
 (where if $j =i$ there is no Koszul complex)
 satisfies  conditions (1) and (2)  of Proposition~\ref{KAScomplex};

\item $c_{k}^t\Tor{1}{R/(I_{j}^n+(c_{j+1}^{t_{j+1}}, \ldots, c_i^{t_i}))}{M}=0$, for all $M \in \mathcal{M}$ and for all $k=i, \dots, d$;

\item $c_{k}^t\left((I_{j}^n+(c_{j+1}^{t_{j+1}}, \ldots, c_i^{t_i}))G \cap N\right) \subseteq \left(I_{j}^n+(c_{j+1}^{t_{j+1}}, \ldots, c_i^{t_i})\right)N$, for all modules $N$ and free modules $G$ such that $N \subseteq G$ and $G/N\in \mathcal{M}$,  and for all $k=i, \dots, d$;
\item  
$c^t_k\left((I_j^n+(c_{j+1}^{t_{j+1}}, \ldots, c_{i-1}^{t_{i-1}}))M: c_i^\infty\right) \subseteq\left(I_j^n+ (c_{j+1}^{t_{j+1}}, \ldots, c_{i-1}^{t_{i-1}})\right)M,$
 for all $k=i, \dots, d$.  
  \end{enumerate}
\end{prp}
\begin{proof} 
For (1), let $S=\mathbb{Z}[X_1, \dots, X_d, C_1, \ldots, C_d]$
 be the polynomial ring in $2d$
 variables and define the ring homomorphism $\psi : S \to R$ 
 such that $\psi(X_j)=x_j$ and $\psi(C_j) = c_j$.
Denote by $\underline{X_j}$ the sequence $X_1, \dots, X_j$. 
By Remark \ref{EaNo}, the complex $B^{\underline{X_j}, n}$ 
is a free resolution over $S$ of the ideal $(X_1, \dots, X_j)^n$.  
Since the elements $C_1,\ldots, C_d$ are  a regular sequence over
 the $S$-module $H_0(B^{\underline{X_j}, n}) = S/ (X_1,\ldots, X_j)^n$, the complex
  $B^{\underline{X_j}, n} \otimes_S K_\bullet(C_{j+1}^{t_{j+1}}, \ldots, C_i^{t_i}; S)$ is a  free resolution of 
  $S/((X_1, \ldots, {X_j})^n, C_{j+1}^{t_{j+1}}, \ldots, C_i^{t_i})$ and therefore  it satisfies the standard 
  conditions on rank by \ref{eisbuc}.
 By induction on the length of the Koszul complex and using \ref{EaNo}, one can see 
 that each ideal of rank-size minors is, up to  radical, the ideal
  $(X_1,\ldots, X_j, C_{j+1}, \ldots, C_i)$.
 Base change to $R$ clearly preserves the rank in this 
 case, as well as condition (2) of Proposition~\ref{KAScomplex}.
%As $B^{\underline{X}_i,n}=B^{\underline{x}_i,n} \otimes R$
%, we obtain that  $B^{\underline{x}_i,n}$ satisfies the standard condition on rank and height.

For (2), denote by $F_{\bullet}$ a minimal  free resolution of  
$R/(I_{j}^n+(c_{j+1}^{t_{j+1}}, \ldots, c_i^{t_i}))$. 
The complexes $F_{\bullet}$ and $B^{I_{j}, n} \otimes K_\bullet(c_{j+1}^{t_{j+1}},\ldots, c_i^{t_i })$ are identical up to homological degree $1$ 
and up to a change of base. Therefore, as $F_{\bullet}$
 is exact, there is a chain map from 
$ B^{I_{j}, n} \otimes K_\bullet(c_{j+1}^{t_{j+1}}, \ldots, c_i^{t_i})$ to $F_{\bullet} $ that  lifts 
the identity between the free modules of homological degree $1$. This chain map induces the following surjection
\begin{equation*}
\ho_1(B^{I_{j}, n} \otimes K_\bullet(c_{j+1}^{t_{j+1}}, \ldots, c_i^{t_i} ;R)\otimes_R M) \rightarrow 
\Tor{1}{R/(I_{j}^n+(c_{j+1}^{t_{j+1}}, \ldots, c_i^{t_i}))}{M}.
\end{equation*}
So it is enough to show that 
\begin{eqnarray}\label{xx}
c_k^t \ho_1(B^{I_{j}, n} \otimes K_\bullet(c_{j+1}^{t_{j+1}}, 
\ldots, c_i^{t_i};R)\otimes_R M)=0.
\end{eqnarray}
 This follows from part (1) and Proposition~\ref{KAScomplex}.

Statement (3) follows from (2) by isomorphism \ref{torUAR}, for $i=1$.

For (4), denote the first and the second maps of the complex
$B^{\underline{X_j},n} \otimes K_\bullet(C_{j+1}^{t_{j+1}}, \ldots, C_i^{t_i}; S)$
  by $\Phi_1$ and $\Phi_2$. 
   Such complex is a free resolution of $S/((X_1, \ldots, {X_j})^n, C_{j+1}^{t_{j+1}}, \ldots, C_i^{t_i})$ 
 and therefore the entries of the last row of a matrix representing $\Phi_2$
generate the ideal $((X_1, \ldots, {X_j})^n, C_{j+1}^{t_{j+1}}, \ldots,
  C_{i-1}^{t_{i-1}}: C_i^{t_i})$ which is equal to $((X_1, \ldots, {X_j})^n,  C_{j+1}^{t_{j+1}}, 
  \ldots, C_{i-1}^{t_{i-1}} )$ for each value of $t_i$ since the elements $X_1, \dots, X_j, C_{j+1}, 
  \dots, C_i$ form a regular sequence.  

Let $\mathbf{z} \in \left((I_j^n+(c_{j+1}^{t_{j+1}}, \ldots, c_{i-1}^{t_{i-1}}))M: c_i^\infty\right)$, 
there exists a tuple $\underline{\mathbf{z}} \in \Ker ( \Phi_1 \otimes_S \id_M)$, for which  
$\mathbf{z}$ is the last entry.
Notice that equation \ref{xx} reads
\begin{eqnarray*}
c_k^t \Ker (\Phi_1 \otimes_R \id_M) \subseteq \Im( \Phi_2 \otimes_R\id_M ).
\end{eqnarray*}

 Therefore,  
$ c^t_k \underline{\mathbf{z}} \in \im (\Phi_2 \otimes_S \id_M) =\im(\Phi_2) \otimes_S \id_ M$. 
By reading this inclusion componentwise, we obtain that 
$$c_k^t \mathbf{z} \in ((X_1, \ldots, {X_j})^n,  C_{j+1}^{t_{j+1}}, 
\ldots, C_{i-1}^{t_{i-1}} )\otimes_S M=(I_j^n+ (c_{j+1}^{t_{j+1}}, \ldots, c_{i-1}^{t_{i-1}}))M,$$
which finishes the proof of the proposition.
\end{proof}

%%%%%%%%%%%%%%%%%%%%%
%%%%%%%%%%
%%%%%%%%%%
\section{Proof of the Main Theorem}
%%%%%%%%%%%%%%%%%%%%%%
%%%%%%%%%%%%%%%%%%%
%%%%%%%%%%%%%%%%%%%%%%%%
%%%%%BRIANCON SKODA
%%%%%%%%%%%%%%%%%%%%%%%%%
%%%%%%%%%%%%%%%%%%%%%%%%%%
In this section, $\Rm$ will be a local ring of dimension $d$. 
If $(S,\n)$ is a faithfully flat local extension of $R$ then the conclusion of the Main Theorem descends from $S$ to $R$.  Hence we may always assume that $R$ has an infinite residue field.
The idea behind the proof of the Main Theorem is to first pick a KAS, then reduce 
the statement to families of $\m$-primary ideals generated by minimal reductions (which are then systems of parameters) and then, by picking a generic generating set for the reduction, change the calculation from the reduction  to the KAS.

Recall that given an ideal $I \subseteq R$, a {\it reduction} of $I$ is an ideal $J \subseteq I$ such that
there exists an integer $k$ satisfying the equality $I^{k+1}=JI^k$, which in turn implies that 
$I^n \subseteq J^{n-k}$, for all $n\ge k$. If $J \subseteq I$ is a reduction of $I$ then 
$\overline{I} = \overline{J}$. When $(R,\m)$ has an infinite residue field, minimal reductions of an ideal always exist.  For a comprehensive 
treatment one may consult \cite{HuSw-06}.

We will rely heavily on Huneke's Uniform Brian\c con-Skoda Theorem
 (Theorem 4.13 in \cite{Hune-92}).  
 We present it in a less general form which is adequate for our needs.
\begin{ipg} \label{UBS}
  Let $\Rmk$ be a complete reduced Noetherian ring.  
  There exists a positive integer $k$ such
   that for all ideals $I \subseteq R$, the inclusion
  $\overline{I^n} \subseteq I^{n-k}$ holds for all $n\ge k$.
\end{ipg}

%We need to use a version of this theorem in the case of non-reduced rings.  
%The statement below can be easily generalized to any of the 
%conditions where Theorem 4.13 of \cite{Hune-92} holds, 
%but we state only the case we need.

Our next theorem shows that the integer $k$ in the definition of a reduction of an ideal
can be uniformly chosen to work for all ideals and for all reductions.

\begin{thm}\label{extendedUBS}
 Let $\Rmk$ be a complete local ring.  
 There exists a positive integer $k$ such that for 
 all ideals $I$ and all reductions $J \subseteq I$, the inclusion
  $I^n \subseteq J^{n-k}$ holds for all $n \ge k$.
\end{thm}

\begin{proof} 
 Denote by $N$ the nilradical of the ring $R$. If $N =0$, we are done by Theorem \ref{UBS}, 
since $I^n \subseteq \overline{J^n}$.  By Noetherian induction, for any non-zero element
$x \in N$, the result holds, since the property of being a reduction 
is preserved modulo nilpotent elements.  We can assume that $x^2=0$.  
Let $k_1$ be an integer which works mod $xR$, and let $k_2$ be 
the UAR number for $xR \subseteq R$.  
Then $I^n \subseteq J^{n-k_1} + (xR \cap I^{n-k_1}) \subseteq J^{n-k_1} + x I^{n-k_1-k_2} \subseteq
J^{n-k_1} + x(J^{n-2k_1-k_2} + xR) \subseteq J^{n-2k_1-k_2}$.
\end{proof}

We need a notion of reduction that relates to the notion of $\kas$ sequences.

\begin{dfn}\label{specialReduction} Let $I$ be an $\m$-primary ideal
 and $\cc=c_1, \dots, c_d$ be  a $\kas$ sequence.
A special reduction of $I$ with respect to $\cc$ is a sequence $x_1, \dots, x_d$ which is well-suited to $\cc$ 
and verifies the following, for all integers $i$ such that $0 \le i \le d-1$:
\begin{enumerate}
\item $I_{d-i-1}$ is a reduction of $I_{d-i}$ modulo $(c_{d-i},  \dots, c_d)$
\item $I_{d-i-1}$ is a reduction of $I_{d-i}$ modulo $(0:(0:c_{d-i}))$,
\end{enumerate}
where $I_k$ denotes the ideal generated by $x_1, \dots, x_k$.  We also set $I_0 = 0$.
\end{dfn}

\begin{rmk} Special reductions exist if the residue field is infinite.
 We need to pick the $\kas$ sequence such that $\dim R/(0:(0:c_{d-i})) \le d-i-1$, 
 and this is possible by Theorem \ref{KASexistence}.   Suppose that we have picked a KAS, $\cc = c_1,\ldots, c_d$.   There are only finitely many ideals generated by the subsets of $\{c_1,\ldots, c_d\}$.  Thus, with an infinite field, one can choose  $d$ general generators of the ideal $I$  that will both be well-suited to $\cc$ and satisfy the reduction conditions (1) and (2) of Definition~\ref{specialReduction}.
\end{rmk}

We now collect the steps that will reduce the proof of the Main Theorem for the family of all ideals to the family of ideals generated by sequences that are  special reductions.

 \begin{ipg}\label{facts} The following statements hold:
 \begin{enumerate}
\item If the Main Theorem holds for the family of all $\m$-primary ideals then it will hold for the family of all ideals,
as  $I= \displaystyle \bigcap_{n \ge 0}(I+\m^n) $ by the Krull Intersection Theorem. 
\item In view of Theorem \ref{extendedUBS}, we can reduce the proof of our Main Theorem to a family of ideals that are special reductions of $\m$-primary ideals. 
\end{enumerate}
\end{ipg}

We now present the technique that will allow to use an induction on the number of elements of the sequences.
Before we need a technical lemma.

  \begin{lem} \label{celimination}
Let $\Rm$ be a local ring and let $N$ be a $d$th syzygy.
Let $\cc=c_1, \dots, c_d$  be  a $\kas$  sequence with 
respect to the family of modules that are $d$th syzygies, and 
let $t$ be the integer given by Proposition~\ref{KAScomplex} for $\cc$.
 Assume that the sequence of elements $x_1, \dots, x_d$ is well-suited to $\cc$.
Let $\mathbf{w}$ be an element in $N$.
If there exists an integer $h$ such that
$$ c^t_j \mathbf{w} \in (x_1,\dots,x_{j-1})^{n-h}N + (c_i^{(i-j+1)t}, \dots, c_{j+1}^{2t})
(x_1,\dots,x_j)^{n-h}N,$$
where $i>j$ and $ 1\le  i \le d$, 
then there exists 
$$\mathbf{w}_1 \in (x_1,\dots,x_j)^{n-h}N$$
such that 
 $$c^t_j(\mathbf{w}-\mathbf{w}_1) \in (x_1,\dots,x_{j-1})^{n-h}N + (c_{i-1}^{(i-j)t}, \dots , c_{j+1}^{2t}\big)(x_1,\dots,x_j)^{n-h}N.$$
\end{lem}

\begin{proof}
For notational simplicity, replace each $c_k^t$ by $c_k$, and let $I_k$ be the ideal $ (x_1,\ldots, x_k)$ for every integer $k$ such $1 \le k \le d$.

 Let ${\bf v}_i \in (x_1,\dots,x_j)^{n-h}N$ be the coefficient of $c_i^{i-j+1}.$
 Then 
 \begin{equation*}
  c_i^{i-j+1}{\bf v}_i  \in  c_jN +I_{j-1}^{n-h}N+ (c_{i-1}^{i-j}, \ldots,  c_{j+1}^2) I_j^{n-h}N.
 \end{equation*}
 In particular, 
 ${\bf v}_i \in \left(I_{j-1}^{n-h}N + (c_{i-1}^{i-j}, \dots, c_{j+1}^2, c_j)N\right):c_i^{i-j+1}$.
% and
% $${\bf v}_i \in \left((x_1,\dots, x_{j-1})^{n-h} +(c_{i-1}^{i-j}, \dots, c_{j+1}^2, c_j )\right)N: c_i^{i-j+1}.$$ 
By Proposition~\ref{fromAGT} (4),
 \begin{equation}\label{eqn1}
 c_i {\bf v}_i \in I_{j-1}^{n-h}N + (c_{i-1}^{i-j}, \dots,  c_{j+1}^2, c_j) N.
 \end{equation} 

We  claim that for every integer $r$ such that $1\le r\le i-j+1$, we have 
\begin{equation}\label{eqn2}
c_i^r {\bf v}_i \in I_{j-1}^{n-h}N + (c_{i-1}^{i-j}, \ldots, c_{i-r+1}^{i-j-r+2})I_j^{n-h}N+ 
  (c_{i-r}^{i-j-r+1}, \dots , c_{j+1}^2 , c_j) N,
\end{equation}
and we will prove  the claim by induction on $r$. 

The claim is true for $r=1$ by \ref{eqn1}.
 Assume the claim is true for $r\ge 1$. 

Let ${\bf n}_{i-r} \in N$ be the coefficient of $c_{i-r}^{i-j-r+1}$ 
in \ref{eqn2}. Then 
\begin{equation*}
{\bf n}_{i-r} \in \left(I_j^{n-h}+
(c_{i-r-1}^{i-j-r},\dots ,c_{j+1}^2 , c_j)  \right)N:c_{i-r}^{i-j-r+1}.
\end{equation*} 
By Proposition~\ref{fromAGT} (4),
 \begin{equation*}
 c_i{\bf n}_{i-r} \in \left(I_j^{n-h}+ 
(c_{i-r-1}^{i-j-r},\dots ,c_{j+1}^2 , c_j)  \right)N.
 \end{equation*} 
 Multiplying \ref{eqn2} by $c_i$ and substituting for $c_i \mathbf{n}_{i-r}$ yields
\begin{equation*}
c_i^{r+1}{\bf v}_i \in I_{j-1}^{n-h}N + (c_{i-1}^{i-j},\dots, c_{i-r+1}^{i-j-r+2}, c_{i-r}^{i-j-r+1})I_j^{n-h}N +
( c_{i-r-1}^{i-j-r} ,
 \dots, c_{j+1}^2, c_j )N,
 \end{equation*}
  as desired. 
Therefore, the claim is true for all $r=1,\dots, i-j+1$. 

In particular, for $r=i-j+1$, we obtain 
\begin{eqnarray*}
c_i^{i-j+1}{\bf v}_i \in I_{j-1}^{n-h}N +
 (c_{i-1}^{i-j}, \dots, c_j)I_j^{n-h}N.
 \end{eqnarray*}
By hypothesis, we have the following containment
\begin{eqnarray*}
c_j {\bf w} \in I_{j-1}^{n-h}N + c_i^{i-j+1} {\bf v}_i + (c_{i-1}^{i-j},\dots,c_{j+1}^2)I_j^{n-h}N,
\end{eqnarray*}
 which, together with the previous containment, gives
 \begin{eqnarray*}
c_j {\bf w} \in I_{j-1}^{n-h}N + (c_{i-1}^{i-j}, \dots, c_{j+1}^2, c_j)I_j^{n-h}N.
\end{eqnarray*}
We conclude that there exists ${\bf w}_1 \in I_j^{n-h}N$ such
 that 
 \begin{equation*}
 c_j({\bf w}-{\bf w}_1) \in I_{j-1}^{n-h}N + (c_{i-1}^{i-j}, \dots, c_{j+1}^2)I_j^{n-h}N,
   \end{equation*}
   which concludes the proof.
\end{proof}

We will now focus our attention on choosing the sequence 
$x_1, \dots, x_d$ to which we can apply the previous lemma.

With the same notation as in \ref{specialReduction}, we can state the main reduction that is used to prove the main theorem.

\begin{prp}\label{mainreduction}
Let $(R,\m,\k)$ be a local ring of dimension $d$ with an
 infinite residue field. Let
 $\cc=c_1,\dots, c_d$ be a $\kas$-sequence  of $R$,  and
   let  $0 \to N \to G \to M \to 0$ be a short exact sequence
 of finitely generated $R$-modules where $M$ is a $d$th syzygy and $G$ is a free $R$-module. 
 
 There exists an integer $h$, depending only on 
 $\cc$, such that if $I$ is an $\m$-primary ideal and $x_1,\dots, x_d$
  is a special reduction  of $I$,  then for all integers $i$ such that $0 \le i \le d-1$ and for all $ n \ge h$, the following inclusion holds
$$ I_{d-i}^n G \cap N \subseteq I_{d-i}^{n-h}N + I_{d-i-1}^{n-h}G \cap N.$$
\end{prp}

\begin{proof} 
Let $t$ be the exponent for $\cc$ given by Proposition~\ref{KAScomplex}.  
For simplicity of notation, we replace each $c_i^t$ by $c_i$, for each integer $i$ such that 
 $1 \le i \le d$.

By Proposition~\ref{fromAGT} (3) with $i=j$, $c_{d-i} (I_{d-i}^n G \cap N) \subseteq I_{d-i}^n N.$ 
In particular, given an  element ${\bf w} \in I_{d-i}^n G \cap N$, we obtain  $c_{d-i} {\bf w} \in I_{d-i}^n N.$ 
Since $I_{d-i-1}$ is a reduction of $I_{d-i}$ modulo the ideal $(c_d^{i+1},\dots, c_{d-i+1}^2, c_{d-i})$, 
 there exists an integer $h_1$, which, by Proposition \ref{extendedUBS},
  depends only on the $\kas$ sequence $\cc$, such that
\begin{eqnarray}\label{eq1}
I_{d-i}^n \subseteq I_{d-i-1}^{n-h_1} + (c_d^{i+1},\dots, c_{d-i+1}^2, c_{d-i})R \cap I_{d-i}^{n-h_1},
\end{eqnarray}
for all $n \ge h_1$.

By the uniform Artin-Rees property,  there exists an integer $h_2$,  
depending only on the $\kas$ sequence $\cc$, such that 
$$
 (c_d^{i+1},\dots, c_{d-i+1}^2, c_{d-i})R \cap I_{d-i}^{n-h_1}\subseteq  (c_d^{i+1},\dots, c_{d-i+1}^2, c_{d-i}) I_{d-i}^{n-h_1-h_2}.
 $$
 Combining this last inclusion with inclusion \ref{eq1}, and setting $h_3$ equal to $h_1+h_2$, we obtain 
$$I_{d-i}^n \subseteq I_{d-i-1}^{n-h_3} + (c_d^{i+1},\dots, c_{d-i+1}^2, c_{d-i})I_{d-i}^{n-h_3},$$
for all $n \ge h_3$.
Thus the containment $c_{d-i} {\bf w} \in I_{d-i-1}^{n-h_3}N + (c_d^{i+1},\dots, c_{d-i+1}^2, c_{d-i})I_{d-i}^{n-h_3}N$  holds. It follows that 
 there exists an element ${\bf w}_1 \in I_{d-i}^{n-h_3}N$  such that 
$$c_{d-i}({\bf w}-{\bf w}_1) \in I_{d-i-1}^{n-h_3} N + (c_d^{i+1},\dots, c_{d-i+1}^2)I_{d-i}^{n-h_3}N .$$

By  a repeated application of Lemma \ref{celimination}, 
 there are elements  ${\bf w}_2, \ldots, {\bf w}_{i+1} \in I_{d-i}^{n-h_3}N$ satisfying 
$$c_{d-i} ({\bf w}-{\bf w}_1-{\bf w}_2-\cdots- {\bf w}_{i+1}) \in I_{d-i-1}^{n-h_3}N 
\subseteq I_{d-i-1}^{n-h_3}G \cap c_{d-i} G. $$
By the uniform Artin-Rees property, there exists an integer  $h_4 \ge h_3$, 
 depending only on the element $c_{d-i}$,
 such that $I_{d-i-1}^{n-h_3}G \cap c_{d-i} G \subseteq c_{d-i}I_{d-i-1}^{n-h_4} G$.  
 So we obtain
$$ c_{d-i} ({\bf w}-{\bf w}_1-{\bf w}_2-\dots- {\bf w}_{i+1}) \in c_{d-i}I_{d-i-1}^{n-h_4} G,$$
for all $n \ge h_4$.

Therefore there exists an element  ${\bf f} \in I_{d-i-1}^{n-h_4}G$ such that 
$$c_{d-i} ({\bf w}-{\bf w}_1-{\bf w}_2-\dots-{\bf w}_{i+1}-{\bf f})=0,$$
and then  ${\bf w}-{\bf w}_1-{\bf w}_2-\dots-{\bf w}_{i+1}={\bf f}+{\bf z},$ where ${\bf z} \in (0:_G c_{d-i})G.$ 
It follows that  
$${\bf z} \in I_{d-i}^{n-h_4}G \cap (0:_G c_{d-i}).$$
 By the uniform Artin-Rees property, there exists an integer $h_5 \ge h_4$, which depends only on the element $c_{d-i}$,  such that
 $ I_{d-i}^{n-h_4}G \cap (0:_G c_{d-i})\subseteq I_{d-i}^{n-h_5} (0:_G c_{d-i})$. 
Thus ${\bf z} \in I_{d-i}^{n-h_5} (0:_G c_{d-i})$, for all $n \ge h_5.$\\

As $I_{d-i-1}$ is a reduction of $I_{d-i}$ modulo $(0:_R(0:_Rc_{d-i})),$ 
there exists an integer $h_6$, depending only on 
$c_{d-i}$, such that   $I_{d-i}^{n} \subseteq I_{d-i-1}^{n-h_6} + (0:_R(0:_Rc_{d-i}))$, for all $n \ge h_6.$
 Thus, by setting $h_7$ equal to $h_5+h_6$, we obtain
$$I_{d-i}^{n-h_5} (0:_G c_{d-i}) G \subseteq I_{d-i-1}^{n-h_7} (0:_G c_{d-i}) G \subseteq I_{d-i-1}^{n-h_7} G,$$
which implies that  ${\bf f}+{\bf z} \in I_{d-i-1}^{n-h_7}G.$ 

As  ${\bf w}-{\bf w}_1-\dots-{\bf w}_{i+1} \in N$
 and ${\bf w}-{\bf w}_1-\dots-{\bf w}_{i+1} = {\bf f}+{\bf z}$, we have 
${\bf w}-{\bf w}_1-\dots-{\bf w}_{i+1} \in I_{d-i-1}^{n-h_7}G \cap N.$
Thus, ${\bf w} \in I_{d-i}^{n-h_7}N + I_{d-i-1}^{n-h_7}G \cap N$ as desired.
\end{proof}

We are finally ready to prove the theorem, from which the Main Theorem in the introduction follows as a corollary.

\begin{thm}\label{mainthm}
Let $(R,\m, k)$ be a local ring of dimension $d$.
  There exists an integer $h$ such that for all ideals 
  $I \subseteq R$ and for all short exact sequences of finitely
   generated modules $0 \to A \to B \to M \to 0$ with $M$ is a
    $d$th syzygy,
\begin{equation*}
I^n B \cap A \subseteq  I^{n-h} A.
\end{equation*}
\end{thm}

\begin{proof}
If there exists such a bound for any faithfully flat extension of $R$ then the same bound holds for $R$.  
Thus, without loss of generality, we may assume that $R$ is complete and has an infinite residue field. 

We can reduce to the case that the middle module is free as follows.  Let $G$ be a free module mapping onto $B$ and let $N$  be the kernel of the composite map from $G$ to $M$.  Then $G/N \cong M$ and $N$ maps to $A$ (via the snake lemma).  If $w \in I^nB \cap A$, then we may lift $w$ to $w_1 \in I^nG$.  Since $w$ maps to $0$ in $M$, so does $w_1$, i.e., $w_1 \in I^n G \cap N$.  So, if $w_1 \in I^{n-h}N$ then $w \in I^{n-h}A$.  For the rest of the proof, we keep the notation that $G$ is free and $G/N \cong M$.

Fix  a $\kas$ sequence $\cc$  for the class of finitely generated $d$th syzygies of $R$,
as given by Theorem~\ref{KASexistence}. By \ref{facts}, it is enough to find
 an integer  $h$ such that  $J^nG \cap N \subseteq J^{n-h}N$ for all ideals $J$ that are special reductions of $\m$-primary ideals with respect to the $\kas$ sequence $\cc$.
 
 Let $I= I_d=(x_1, \dots, x_d)$ be a special reduction.  
 Let $k$ be the integer given by Proposition~\ref{mainreduction}.
%  let  $h''$ be the integer given by Proposition~\ref{principal}, and set $k=\max\{h', h''\}$.
 By  $d$ applications of Proposition~\ref{mainreduction}, %  together with Proposition~\ref{principal}, we show that
  
\begin{align*}
I^n G \cap N& \subseteq I^{n-k}N +( I_{d-1}^{n-k}G \cap N) \subseteq \cdots \\
  & \subseteq I^{n-k}N + I_{d-1}^{n-2k}N + \cdots + I_2^{n-(d-1)k}N+\left( I_1^{n-(d-1)k}G \cap N\right) \subseteq I^{n- dk}N.
\end{align*}
Hence $h = dk$ suffices.
\end{proof}

\begin{cor}\label{syzAR}
Let $(R,\m,k)$ be a local ring.  Then every finitely generated module is syzygetically Artin-Rees with respect to every ideal in $R$. 
\end{cor}
\begin{proof}
By Theorem~\ref{mainthm}, there exists an integer $h$ such that if $I$ is any ideal and if ($F_\bullet, \partial_\bullet)$ is a free resolution of $M$, then $I^n F_i \cap \partial_{i+1}(F_{i+1}) \subseteq I^{n-h} \partial_{i+1}(F_{i+1})$ for $i \ge d = \dim R$.  Once $I$ is fixed, then there are also  Artin-Rees numbers for the earlier syzygies, so the maximum of $h$ and these numbers works for $I$ and $M$.
\end{proof}

A second corollary has to do with perturbations of resolutions.  Let $(R,\m)$ be a local ring.  A complex $G'_\bullet$ is a perturbation of $G_\bullet$ to orders $q_1, q_2, \ldots $ if the free modules in each are the same and the difference of the $n$th differentials maps $G_n$ to $\m^{q_n} G_{n-1}$.  Eisenbud and Huneke raise the following question (Question C in \cite{EisHun-02}):  If $G_\bullet$ is a minimal free resolution of the finitely generated module $M$, is there a number $q$ such that any complex $G'_\bullet$ which is a perturbation of $G_\bullet$ to order $q, q, \ldots $ is exact?

\begin{cor}
Let $(R,\m)$ be a local ring of dimension $d$.  Question C has an affirmative answer.  Moreover, for the class of $d$th syzygy modules, the integer $q$ may be chosen depending only on the ring.
\end{cor}
\begin{proof}
We refer the reader to Proposition 1.1 of \cite{EisHun-02} for details of how the  syzygetic  Artin-Rees property is connected to the perturbation question.
\end{proof}

%\begin{thm} Let $(\Rmk)$ be a local notherian ring. Every finitely generated $R$-module $M$ is syzygetically Artin-Rees with respect to the family of $\m$-primary ideals.
%\end{thm}

%%\section*{Acknowledgments}

%%% BIBLIOGRAPHY
\bibliographystyle{amsplain}

\end{document}